\documentclass[12pt]{amsart}

\usepackage[utf8]{inputenc}
\usepackage[T1,T2A]{fontenc}
\usepackage[english]{babel}
\usepackage{amsfonts}
\usepackage{amsbsy}
\usepackage{amssymb}
\usepackage{fixltx2e,graphicx}
\usepackage{mathrsfs}
\usepackage[colorlinks,breaklinks,allcolors=blue]{hyperref}
\usepackage{longtable}
\usepackage{enumitem}
\usepackage{comment}
\usepackage[cmtip,matrix,arrow]{xy} \SelectTips{cm}{10}
\newcommand{\cxymatrix}[1]{\vcenter{\xymatrix@=15pt{#1}}}

\catcode`~=11 
\newcommand{\urltilde}{\kern -.15em\lower .7ex\hbox{~}\kern .04em}
\catcode`~=13 

\makeatletter

\renewcommand{\@makecaption}[2]{
\vspace{\abovecaptionskip}%
\sbox{\@tempboxa}{#1. #2}%
\global\@minipagefalse \hbox to \hsize {{\scshape \hfil #1.
#2\hfil}} \vspace{\belowcaptionskip}}

\makeatother

\newcommand{\Ker}{\operatorname{Ker}}

\newcommand{\Aut}{\operatorname{Aut}}

\newcommand{\ZZ}{\mathbb Z}
\newcommand{\QQ}{\mathbb Q}

\newcommand{\PP}{\mathbb P}
\newcommand{\KK}{\mathbb K}
\newcommand{\GG}{\mathbb G}

\newtheorem{theorem}{Theorem}

\newtheorem{proposition}[theorem]{Proposition}
\newtheorem{lemma}[theorem]{Lemma}
\newtheorem{corollary}[theorem]{Corollary}

\newtheorem*{question*}{Question}

\theoremstyle{definition}

\theoremstyle{remark}

\newtheorem{remark}[theorem]{Remark}

\makeatletter

\@addtoreset{theorem}{section}

\makeatother

\numberwithin{equation}{section}

\makeatletter

\@namedef{subjclassname@2020}{%
  \textup{2020} Mathematics Subject Classification}

\makeatother

\begin{document}

\date{}
\title[Connecting orbits in quasiaffine spherical varieties]{Connecting orbits in quasiaffine spherical varieties\\ via $B$-root subgroups}
\thanks{The second author was supported by the Russian Science Foundation, grant no.~25-21-00358}
\author{Roman Avdeev}
\address{%
{\bf Roman Avdeev} \newline HSE University}
\email{suselr@yandex.ru}
\author{Vladimir Zhgoon}
\address{%
{\bf Vladimir Zhgoon} \newline Moscow Institute of Physics and Technology (National Research University)}
\email{zhgoon@mail.ru}



\begin{abstract}
Given a connected reductive algebraic group~$G$ with a Borel subgroup~$B$ and a quasiaffine spherical $G$-variety $X$, we prove that every $G$-orbit $Y$ contained in the regular locus of $X$ can be connected by a $B$-normalized additive one-parameter group action with any minimal $G$-orbit in~$X$ containing $Y$ in its closure. As a consequence, we show that the regular locus of~$X$ is transitive for the subgroup in the automorphism group of~$X$ generated by~$G$ and all $B$-normalized additive one-parameter subgroups.
\end{abstract}

\maketitle


\section{Introduction}

Let $X$ be an irreducible algebraic variety over an algebraically closed field $\KK$ of characteristic zero.
Every nontrivial action on~$X$ of the group $\GG_a = (\KK,+)$ induces a one-parameter subgroup in the automorphism group $\Aut(X)$, called a \textit{$\GG_a$-subgroup}.
If in addition $X$ is equipped with an action of an algebraic group~$F$ and $R$ is a $\GG_a$-subgroup normalized by~$F$, then we say that $R$ is an \textit{$F$-root subgroup} on~$X$.

The variety $X$ is said to be \textit{spherical} if it admits a regular action of a connected reductive algebraic group~$G$ such that a Borel subgroup $B \subset G$ has an open dense orbit in~$X$.
In this situation, $X$ automatically contains an open $G$-orbit; moreover, there are only finitely many $G$-orbits in~$X$.
In the recent years, a study of $B$-root subgroups on spherical varieties was initiated in the papers~\cite{AA,AZ1,AZ2}.
One motivation for studying $B$-root subgroups is related to the problem of describing the automorphism groups for complete spherical varieties. In this case, the connected component of the group $\Aut(X)$ is algebraic, so a description of $B$-root subgroups on~$X$ enables one to describe the $G$-module decomposition of the Lie algebra of $\Aut(X)$.
Indeed, every $B$-root subgroup induces a $B$-semiinvariant vector field on~$X$, which is a highest weight vector of a simple $G$-module in the Lie algebra of~$\Aut(X)$.

The main goal of this note is to prove the following theorem, which provides another motivation for studying $B$-root subgroups on spherical varieties.

\begin{theorem} \label{thm_reg_locus}
Let $X$ be a quasiaffine spherical $G$-variety \textup(not necessarily normal\textup) and let $Y$ be a $G$-orbit contained in the regular locus of~$X$. Then for every minimal $G$-orbit $Y' \subset X$ containing $Y$ in its closure there exists a $B$-root subgroup on $X$ that connects $Y$ with~$Y'$.
\end{theorem}

This theorem is a consequence of Theorem~\ref{thm_main} below.

It follows from Theorem~\ref{thm_reg_locus} that, for a quasiaffine spherical $G$-variety~$X$, every $G$-orbit in the regular locus of~$X$ can be connected with the open $G$-orbit via a suitable chain of $B$-root subgroups, which yields the following corollary.

\begin{corollary} \label{crl_reg_locus}
Let $X$ be a quasiaffine spherical $G$-variety \textup(not necessarily normal\textup) and let $\widetilde G$ be the subgroup in $\Aut(X)$ generated by the images of $G$ and all $B$-root subgroups on~$X$.
Then the regular locus of $X$ is a single $\widetilde G$-orbit.
\end{corollary}

In the case where $G = B = T$ and $X$ is affine (so that $X$ is an affine toric $T$-variety), Theorem~\ref{thm_reg_locus} and Corollary~\ref{crl_reg_locus} were known from~\cite{AKZ}.
  
This note may serve as a complement and further development of one of the main results of the preprint~\cite{Sha} (Theorem~3), which states that the regular locus of an affine spherical $G$-variety $X$ is a single orbit of the subgroup in $\Aut(X)$ generated by~$G$ and all $\GG_a$-subgroups. In fact, Corollary~\ref{crl_reg_locus} gives an affirmative answer to~\cite[Question~1]{Sha}. We remark that the strategy of our proof of Theorem~\ref{thm_reg_locus} goes the same lines as that used in~\cite{Sha}.

\textbf{Acknowledgement.}
We thank Anton Shafarevich for sharing with us the results of his preprint~\cite{Sha} before it was written; this inspired our work.

\section{Preliminaries}

Given an algebraic variety~$X$, we denote by $X^{\operatorname{reg}}$ the regular locus of~$X$, by $\Aut(X)$ the group of automorphisms of~$X$, and by $\KK[X]$ the algebra of regular functions on~$X$.
Given an algebraic group $F$, its connected component of the identity is denoted by~$F^0$.

Suppose that $X$ is an affine algebraic variety. A derivation $\partial$ of $\KK[X]$ is called \textit{locally nilpotent} (LND for short) if for every $f \in \KK[X]$ there is $n \in \ZZ_{\ge0}$ such that $\partial^n(f) = 0$.
Every LND $\partial$ on~$\KK[X]$ defines a $\GG_a$-action on~$\KK[X]$ via the formula $(t,f) \mapsto \exp(t\partial)f$ ($t \in \KK$, $f \in \KK[X]$) and hence induces a $\GG_a$-subgroup in~$\Aut(X)$. By \cite[\S\,1.5]{Fre} this yields a bijection between LNDs on~$\KK[X]$ modulo proportionality and $\GG_a$-subgroups on~$X$. Moreover, if $X$ is equipped with an action of an algebraic group~$F$, then it is easy to see that $\partial$ is $F$-semiinvariant if and only if the corresponding $\GG_a$-subgroup is $F$-normalized (and hence an $F$-root subgroup).

Below we shall need the following elementary facts.

\begin{lemma}
Let $\partial$ be an LND on an algebra $A$ over~$\QQ$ without zero divisors. Let $a,b \in A$ such that $ab\in \Ker(\partial)$; then either $a\in \Ker(\partial)$ or $b\in \Ker(\partial)$.
\end{lemma}
\begin{proof}
Let $n,m$ be the minimal numbers such that $\partial^{n+1}(a)=\partial^{m+1}(b)=0$. Then by the Leibnitz rule we get
\[
\partial^{n+m} (ab)=\sum_{k=0}^{n+m} \binom{n+m}{k} \partial^{k}(a)\partial^{n+m-k}(b)= \binom{n+m}{n}\partial^{n}(a)\partial^{m}(b)=0.
\] 
This implies the vanishing of either $\partial^{n}(a)$ or $\partial^{m}(b)$, which contradicts the assumption.
\end{proof}

\begin{corollary} \label{crl_f_in_kernel}
Let $\partial$ be an LND on an algebra~$A$ over~$\QQ$ and let $f\in A$ be an invertible element. Then $f \in \Ker(\partial)$.
\end{corollary}

\begin{lemma} \label{extend}
Suppose that $F$ is a connected linear algebraic group, $X$ is an irreducible affine $F$-variety, $f \in \KK[X] \setminus \lbrace 0 \rbrace$ is an $F$-semiinvariant function, $X_f = \lbrace x \in X \mid f(x) \ne 0 \rbrace$ is the corresponding principal open subset of~$X$, and $R$ is an $F$-root subgroup on~$X_f$.
Then there exists an $F$-root subgroup~$\widetilde R$ on~$X$ that preserves~$X_f$ and has on it the same orbits as~$R$.
In addition, $\widetilde R$ may be chosen in such a way that its acts trivially on~$X \setminus X_f$.
\end{lemma}
\begin{proof}
As $f$ is semiinvariant, the set $X_f$ is $F$-stable.
Let $\partial$ be an $F$-semiinvariant LND on~$\KK[X_f]$ corresponding to~$R$ and fix a finite generating set $\{ f_i\}$ for the algebra $\KK[X]$.
For every~$i$, we have $\partial(f_i) = h_i/f^{n_i}$ with $h_i\in \KK[X]$ and $n_i\in \ZZ_{\ge0}$.
Take any integer $N\ge\max\{n_i\}$.
Corollary~\ref{crl_f_in_kernel} yields $\partial(f) = 0$, so $\widetilde{\partial} = f^N\partial$ is also an LND on~$\KK[X_f]$, which is $F$-semiinvariant.
Let $\widetilde R$ be the $F$-root subgroup on~$X_f$ corresponding to~$\widetilde \partial$.
Clearly, $\widetilde\partial(f_i)\in \KK[X]$ for all~$i$, so $\widetilde\partial(\KK[X])\subset \KK[X]$ by the Leibnitz rule.
It follows that $\widetilde \partial$ is an ($F$-semiinvariant) LND on~$\KK[X]$, hence $\widetilde R$ extends to an $F$-root subgroup on the whole~$X$.
Note that for $N \ge\max\{n_i\} + 1$ we actually have $\widetilde\partial(\KK[X])\subset f\KK[X]$, therefore $\widetilde \partial$ vanishes on $\KK[X \setminus X_f]  \simeq \KK[X] / f\KK[X]$ and hence $\widetilde R$ acts trivially on~$X \setminus X_f$.

Let us show that for every point $x \in X_f$ the orbits $R x$ and $\widetilde R x$ coincide.
Note that $f(x) \ne 0$ and let $\mathfrak m \subset \KK[X]$ be the maximal ideal of~$x$.
Taking into account that $f=f(x) \ {\rm mod} \  \mathfrak m$, for every $g \in \mathfrak m$, $t \in \KK$, and the rescaled parameterization $\widetilde t = f(x)^{-N}t$ we have
\begin{multline*}
\exp(-\widetilde t \, \widetilde{\partial}) \exp(t\partial) g =
 \sum\limits_{i=0}^\infty \frac{(-\widetilde t f^N\partial)^i}{i!}(\exp(t\partial) g) \subset \sum\limits_{i=0}^\infty \frac{(-\widetilde t f(x)^N\partial)^i}{i!}(\exp(t\partial) g) + \mathfrak m = \\
\sum\limits_{i=0}^\infty \frac{(-t\partial)^i}{i!}(\exp(t\partial) g) + \mathfrak m = \exp(-t\partial)(\exp(t\partial) g) + \mathfrak m = g + \mathfrak m = \mathfrak m.
\end{multline*} 
It follows that $\exp(-\widetilde t \, \widetilde{\partial}) \exp(t\partial) \mathfrak m \subset \mathfrak m$, but since both expressions are maximal ideals in~$\KK[X]$, we conclude that they coincide.
The latter yields $\exp(t\partial)x=\exp(\widetilde t \, \widetilde{\partial})x$ for all $t \in \KK$, whence $Rx = \widetilde Rx$.
\end{proof}

Let $G$ be a connected reductive algebraic group with a Borel subgroup~$B$ and let $\widetilde X$ be an irreducible quasiaffine $G$-variety.
By~\cite[Theorem~1.6]{PV}, there exists an open $G$-equivariant embedding of $\widetilde X$ in an (irreducible) affine $G$-variety~$X$.
In view of this, in what follows we shall work with irreducible affine varieties $X$ together with a fixed $G$-stable closed boundary subset~$\partial X$ keeping in mind the possibility to extend results to the open quasiaffine subvariety $X \setminus \partial X$.

\begin{lemma}\label{func}
Suppose $X$ is an irreducible affine $G$-variety and $X_0 \subset X$ is a proper $G$-stable closed subvariety.
Let $Y$ be a $G$-orbit in $X \setminus X_0$ and let $D_Y \subset Y$ be a proper $B$-stable closed subvariety.
Then there is a $B$-semiinvariant function $f \in \KK[X]$ vanishing on $X_0 \cup D_Y$ and not vanishing identically on~$Y$.
\end{lemma}
\begin{proof}
Let $\overline Y$ (resp. $\overline D_Y$) be the closure of~$Y$ (resp.~$D_Y$) in~$X$.
Consider the natural $G$-equivariant surjective map $\pi \colon \KK[X] \rightarrow \KK[\overline{Y}]$ given by restricting functions.
Since $Y$ is not contained in $X_0 \cup \overline D_Y$, there is a function $h \in \KK[X]$ that vanishes on $X_0 \cup D_Y$ and not on~$Y$.
Then $\pi(h)$ is a nonzero function in~$\KK[\overline Y]$ vanishing on~$D_Y$, hence all functions in the $B$-submodule $\langle B \pi(h) \rangle \subset \KK[\overline Y]$ vanish on~$D_Y$.
By the Lie--Kolchin theorem, $\langle B \pi(h) \rangle$ contains a nonzero $B$-semiinvariant function, hence changing $h$ we may assume that $\pi(h)$ is $B$-semiinvariant.
Since $G$ is reductive, the $G$-submodule $\langle Gh \rangle \subset \KK[X]$ admits a $G$-module decomposition $\langle Gh \rangle=\left.\Ker \pi\right|_{\langle Gh \rangle} \oplus V$ for some $G$-submodule $V \subset \langle Gh \rangle$.
Consider the expression $h=f_0+f$ with respect to this decomposition.
Clearly, $\pi$ maps $V$ isomorphically to its image $\pi(V) = \pi(\langle Gh \rangle)$, so $f$ is $B$-semiinvariant and hence has all the desired properties.
\end{proof}

\section{Tools}

\subsection{Local structure theorem}
We now state the version of the local structure theorem of Knop \cite[Thm. 2.3, Prop. 2.4]{Kn94} (see also~\cite{BLV}) for affine varieties; see \cite[\S\,2]{Kn98}.

Let $X$ be an irreducible affine $G$-variety (not necessarily normal).
Consider a nonzero $B$-semiinvariant function $f \in \KK[X]$ and the corresponding principal open subset $X_f = \lbrace x \in X \mid f(x) \ne 0 \rbrace$.
Let $P$ be the stabilizer in~$G$ of the line~$\KK f \subset \KK[X]$; this is a parabolic subgroup in~$G$ containing~$B$.
Let $P_u$ be the unipotent radical of~$P$.
Consider also one more parabolic subgroup $P(X) \supset B$ defined as the intersection of stabilizers of all $B$-stable prime divisors in~$X$. 

Consider the following $P$-equivariant map:
\[
\psi \colon X_f \rightarrow \mathfrak g^*, \ x\mapsto l_x,
 \ \text{ where } \ l_x(\xi)=\frac{\xi f}{f}(x).
\]

Fix a point $z \in X_f$; put $l =\psi(z)$,  $L =G_{l}$, and $Z =\psi^{-1}(l)$.

\begin{theorem}
\label{thm_lst}
Under the above assumptions and notation, the following assertions hold.
\begin{enumerate}[label=\textup{(\alph*)},ref=\textup{\alph*}]
\item \label{thm_lst_a}
The image of $\psi$ is a single $P$-orbit isomorphic to $P_u$.
\item \label{thm_lst_b}
$L$ is a Levi subgroup of $P$ and there is a $P$-equivariant isomorphism
\[
P_u \times Z \simeq P*_{L} Z \xrightarrow{\sim} X_f.
\]
\item \label{thm_lst_c}
Assume that $P=P(X)$.
Then the derived subgroup of~$L$ acts trivially on~$Z$.
\end{enumerate}
\end{theorem}

Since all Levi subgroups in~$P$ are conjugate, by changing the point~$z$ within its $P$-orbit we may assume that $L\supset T$.
In this case, $B_L = B \cap L$ is a Borel subgroup of~$L$.

\subsection{Root subgroups on total spaces of linearized vector bundles}

In this subsection, we present a straightforward generalization of the construction used in the proof of~\cite[Corollary~4.5]{AFKKZ}.

Let $F$ be a linear algebraic group, $X$ a smooth irreducible $F$-variety, and $E$ a vector bundle on~$X$ with canonical projection $p \colon E \to X$.

An \textit{$F$-linearization} of~$E$ is an $F$-variety structure on~$E$ such that $p$ is $F$-equivariant and for all $x \in X$, $g \in F$ the induced map of fibers $p^{-1}(x) \to p^{-1}(gx)$, $y \mapsto gy$, is linear.
We say that $E$ is \textit{$F$-linearized} if it is equipped with an $F$-linearization.
In this situation, the space $H^0(X,E)$ of global sections of~$E$ naturally becomes an $F$-module with the action of~$F$ given by the formula $(gs)(x) = g(s(g^{-1}x))$ for all $g \in F$, $s \in H^0(X,E)$, and $x \in X$.

The next result is obtained by a direct check.

\begin{proposition} \label{prop_RS_via_sect}
Suppose that $E$ is $F$-linearized and $s \in H^0(X,E)$ is a nonzero $F$-semiinvariant section of weight~$\chi$.
Then the map $\GG_a \times E \to E$, $(t,y) \mapsto y+ ts(p(y))$, defines an $F$-root subgroup $R$ on~$E$ of weight~$\chi$.
Moreover, the image of the zero section of~$E$ is $R$-unstable.
\end{proposition}

\subsection{Homogeneous fiber bundles}

For details on the material in this subsection, we refer to \cite[\S\,2.1]{Tim}.

Let $F$ be a connected linear algebraic group, $H \subset F$ an algebraic subgroup, and $V$ an irreducible $H$-variety.
Consider the action of $H$ on $F \times V$ given by $h(g,v) = (gh^{-1},hv)$ for all $h \in H$, $g \in F$, and $v \in V$ and let $F *_H V$ denote the respective quotient set.
For every pair $(g,v) \in F \times V$, let $g*v$ denote its image in $F *_H V$.

Suppose that $V$ is covered by $H$-stable quasiprojective open subsets.
Then there is a natural algebraic variety structure on~$F *_H V$ such that
\begin{itemize}
\item
$F *_H V$ is an $F$-variety with respect to the action of $F$ given by $\widetilde g(g*v) = (\widetilde g g)*v$ for all $\widetilde g, g \in F$ and $v \in V$;

\item
the natural map $p \colon F *_H V \to F/H$, $g*v \mapsto gH$, is an $F$-eqiuvariant morphism.
\end{itemize}

Identifying $V$ with a subvariety of $F *_H V$ via the map $v \mapsto e*v$ we see that each fiber $p^{-1}(gH)$ is just~$gV$.
In view of the latter property, $F *_H V$ is said to be the \textit{homogeneous fiber bundle} over $F/H$ with fiber~$V$.

If $V$ is a finite-dimensional $H$-module, then $F *_H V$ becomes an $F$-linearized vector bundle on $F/H$ in a natural way.

For future reference, below we state two further properties of homogeneous fiber bundles.

\begin{proposition} \label{prop_hfb_isom}
Suppose that $X$ is an $F$-variety and there exists an $F$-equivariant morphism $\varphi \colon X \to F/K$.
Then there is a natural $F$-equivariant isomorphism $X \simeq F *_H V$ where $V = \varphi^{-1}(eH)$.
\end{proposition}

\begin{proposition} \label{prop_orb_corr}
The map $O \mapsto FO$ is a bijection between the $H$-orbits in~$V$ and the $F$-orbits in $F *_H V$.
Moreover, this bijection preserves codimensions and inclusion relations between orbit closures.
In particular, $O$ is an open $H$-orbit in~$V$ if and only if $FO$ is an open $F$-orbit in $F *_H V$.
\end{proposition}

\section{Main results}

\subsection{General setup}
\label{ssec_gen_setup}

Let $X$ be an irreducible affine $G$-variety (not necessarily normal) and let $\partial X$ be a proper closed $G$-stable subvariety.

Let $Y$ be a $G$-orbit in~$X \setminus \partial X$, denote by $\overline Y$ its closure in~$X$, let $D_Y \subset Y$ be a proper closed $B$-stable subvariety.
By Lemma~\ref{func} applied to $X_0 = \partial X \cup (\overline Y \setminus Y)$ and~$D_Y$,  there is a $B$-semiinvariant function $f \in \KK[X]$ vanishing on $\partial X \cup (\overline Y \setminus Y) \cup D_Y$ and not vanishing identically on~$Y$.
Put $D = \lbrace x \in X \mid f(x) = 0 \rbrace$, $X_f = X \setminus D$, and $Y_f = Y \setminus D$.
Let $P$ be the stabilizer in~$G$ of the line~$\KK f \subset \KK[X]$.
We now apply Theorem~\ref{thm_lst} to~$f$ and a point $z \in Y_f$; consider the resulting Levi subgroup $L \subset P$ and the closed $L$-stable subvariety $Z = \psi^{-1}(\psi(z)) \subset X_f$.
Put also $Z_Y = Z \cap Y = Z \cap Y_f$; this is a closed $L$-stable subvariety of~$Z$.
Let $N_z$ be the normal space to~$Y$ at~$z$ and let $H$ be the stabilizer of $z$ in~$L$.
Without loss of generality we may assume that $L \supset T$.

The following two statements will be used later in the case where $X$ is spherical.

\begin{proposition} \label{prop_GX-LZ}
The following conditions are equivalent.
\begin{enumerate}[label=\textup{(\arabic*)},ref=\textup{\arabic*}]
\item
$X$ is spherical as a $G$-variety.
\item
$Z$ is spherical as an $L$-variety.
\end{enumerate}
\end{proposition}

\begin{proof}
The claim follows from Theorem~\ref{thm_lst}(\ref{thm_lst_b}).
Indeed, put $B_L = B \cap L$; this is a Borel subgroup of~$L$ satisfying $B = B_LP_u$.
If there is an open $B$-orbit $X' \subset X$, then $X' \subset X_f$ and $X' \cap Z$ is an open $B_L$-orbit in~$Z$.
Conversely, if $Z' \subset Z$ is an open $B_L$-orbit, then $P_uZ'$ is an open $B$-orbit in~$X_f$ and hence in~$X$.
\end{proof}

\begin{proposition} \label{prop_Y_spherical}
Suppose that $Y$ is spherical as a $G$-variety and $D_Y$ is the union of all colors in~$Y$.
Then the following assertions hold.
\begin{enumerate}[label=\textup{(\alph*)},ref=\textup{\alph*}]
\item \label{prop_Y_spherical_a}
$Y_f$ is a single $B$-orbit and $Y_f = Y \setminus D_Y$.
\item \label{prop_Y_spherical_b}
$Z_Y$ is a single $L$-orbit.
\item \label{prop_Y_spherical_c}
The derived subgroup of~$L$ acts trivially on~$Z_Y$.
\end{enumerate}
\end{proposition}

\begin{proof}
Clearly, $Y \setminus D_Y$ is a single $B$-orbit, so $D \cap Y = D_Y$ and we get~(\ref{prop_Y_spherical_a}).
Observe that $Y \setminus D_Y$ is in fact a single $P$-orbit, which yields~(\ref{prop_Y_spherical_b}) by Theorem~\ref{thm_lst}(\ref{thm_lst_b}).
Since the restriction map $\pi \colon \KK[X] \to \KK[\overline Y]$ is a $G$-module homomorphism and $\pi(f) \ne 0$, the lines $\KK f$ and $\KK \pi(f)$ have the same stabilizer in~$G$, which is~$P$.
Observe that $P(Y)$ is the intersection of the stabilizers of all colors in~$Y$, which implies $P \subset P(Y)$.
The inverse inclusion also holds by~\cite[Theorem~3.1]{PV}, so $P = P(Y)$.
Replacing $X$ with $\overline Y$ in the construction preceding Theorem~\ref{thm_lst} and applying part~(\ref{thm_lst_c}) of that theorem, we get~(\ref{prop_Y_spherical_c}).
\end{proof}

\begin{proposition} \label{prop_hfb_lst}
Suppose that $Y \subset X^{\operatorname{reg}}$, $Z_Y$ is a single $L$-orbit, and $Z_Y$ is a unique closed $L$-orbit in~$Z$.
Then the following assertions hold.
\begin{enumerate}[label=\textup{(\alph*)},ref=\textup{\alph*}]
\item \label{prop_hfb_lst_a}
There is an $L$-equivariant isomorphism $Z \simeq L*_{H} N_z$.
\item \label{prop_hfb_lst_b}
There is a $P$-equivariant isomorphism $X_f \simeq P*_{H} N_z$.
\end{enumerate}
\end{proposition}

\begin{proof}
(\ref{prop_hfb_lst_a})
This is implied by the Luna slice theorem; see~\cite[\S\,III.1, Corollary~2]{Lu73} or~\cite[Theorem 6.7]{PV}.

(\ref{prop_hfb_lst_b})
Since $Z_Y$ is a single $L$-orbit, it follows that $Y_f$ is a single $P$-orbit, which is isomorphic to~$P/H$.
Now the claim is obtained by combining part~(\ref{prop_hfb_lst_a}) with Theorem~\ref{thm_lst}(\ref{thm_lst_b}) and Proposition~\ref{prop_hfb_isom}.
\end{proof}

\begin{remark}
In fact, if $Y \subset X^{\operatorname{reg}}$ and $Z_Y$ is a single $L$-orbit, then $Z_Y$ is a unique closed $L$-orbit in~$Z$ if and only if $\KK[N_z]^{H} = \KK$.
\end{remark}

In what follows, we put $\mathcal N = P*_{H} N_z$ for short and regard $\mathcal N$ as a $P$-linearized vector bundle on~$P/H$.
Let $\zeta \subset \mathcal N$ be the image of the zero section.

\begin{theorem} \label{thm_BRS_gen}
Under the assumptions of Proposition~\textup{\ref{prop_hfb_lst}}, there exists a $B$-root subgroup on~$X$ that preserves $X_f$ and $\partial X$ and connects~$Y$ with a $G$-orbit $Y' \ne Y$ containing $Y$ in its closure.
\end{theorem}

\begin{proof}
Since $P/H \simeq P_u \times Z_Y$ is affine, it follows that the space of global sections $H^0(P/H, \mathcal N)$ is nonzero, hence by the Lie--Kolchin theorem it contains a nonzero $B$-semiinvariant section~$s$.
Thanks to Proposition~\ref{prop_RS_via_sect}, $s$ defines a $B$-root subgroup $R$ on~$\mathcal N$ such that $\zeta$ is $R$-unstable.
Recall from Proposition~\ref{prop_hfb_lst}(\ref{prop_hfb_lst_b}) that there is a $P$-equivariant isomorphism $X_f \simeq \mathcal N$ and observe that it takes $Y_f$ to~$\zeta$.
Thus we get a $B$-root subgroup (still denoted by~$R$) on~$X_f$ that connects $Y_f$ with a $P$-orbit $O \subset X_f$ different from~$Y_f$.
Since $Y_f$ is a unique closed $P$-orbit in~$X_f$, it follows that $O$ contains $Y_f$ in its closure.
Then the $G$-orbit $Y' = GO \subset X$ contains $Y$ in its closure and is different from~$Y$.
By Lemma~\ref{extend}, there exists a $B$-root subgroup $\widetilde R$ on the whole~$X$ that preserves $X_f$ and $\partial X$ and has on $X_f$ the same orbits as~$R$.
By construction, $\widetilde R$ connects $Y$ with~$Y'$.
\end{proof}

For every $y \in P/H$, let $\mathcal N_y$ denote the fiber of $\mathcal N$ over~$y$.

\begin{remark}
Since the variety $P/{H}$ is affine, the vector bundle $\mathcal N$ is generated by global sections, that is, for every $y \in P/H$ and $v \in \mathcal N_y$ one can find a section $s\in H^0(P/{H}, \mathcal N)$ such that $s(y) = v$.
This shows that, under the assumptions of Proposition~\ref{prop_hfb_lst}, by a $\GG_a$-action that is not nesessary $B$-normalized one can connect~$Y$ with any $G$-orbit that contains $Y$ in its closure.
\end{remark}

\begin{corollary}
Suppose that $X$ is spherical \textup(not necessarily normal\textup) and $Y \subset X^{\operatorname{reg}}$.
Then there exists a $B$-root subgroup on~$X$ that connects~$Y$ with a $G$-orbit $Y' \ne Y$ containing $Y$ in its closure.
\end{corollary}

\begin{proof}
Choose $D_Y$ to be the union of all colors in~$Y$.
It follows from Proposition~\ref{prop_GX-LZ} that $Z$ is a spherical $L$-variety and hence has a unique closed $L$-orbit.
By Proposition~\ref{prop_Y_spherical}(\ref{prop_Y_spherical_b}), this closed $L$-orbit is~$Z_Y$.
Now the claim follows from Theorem~\ref{thm_BRS_gen}.
\end{proof}

\subsection{The spherical case}

Retain the assumptions and notation of \S\,\ref{ssec_gen_setup} and assume in addition that $X$ is spherical (but not necessarily normal).
From Proposition~\ref{prop_GX-LZ} we know that $Z$ is an affine spherical $L$-variety and hence has a unique closed $L$-orbit.
Choose $D_Y$ to be the union of all colors in~$Y$.
Then $Z_Y$ is a single $L$-orbit by Proposition~\ref{prop_Y_spherical}(\ref{prop_Y_spherical_b}), hence it is a unique closed $L$-orbit in~$Z$.
Then $H$ is a reductive group.
Proposition~\ref{prop_Y_spherical}(\ref{prop_Y_spherical_c}) implies that $H$ contains the derived subgroup of~$L$.
Put $B_H = B \cap H$, so that $B_H^0$ is a Borel subgroup of~$H$.

Recall the $P$-linearized vector bundle $\mathcal N = P *_H N_z$ and the image $\zeta$ of its zero section.

\begin{proposition} \label{prop_B-orb_in_N}
Suppose that $Y \subset X^{\operatorname{reg}}$.
Then all minimal $B$-orbits in~$\mathcal N \setminus \zeta$ are precisely those of nonzero $B_H$-semi\-in\-vari\-ant vectors in~$N_z$.
\end{proposition}

\begin{proof}
Thanks to Proposition~\ref{prop_hfb_lst}(\ref{prop_hfb_lst_b}), there is a $P$-equivariant isomorphism $X_f \simeq \mathcal N$.
Since $H$ contains the derived subgroup of~$L$, it follows that $B$ acts transitively on $P/H$, hence by Proposition~\ref{prop_hfb_isom} there is a $B$-equivariant isomorphism $\mathcal N \simeq B *_{B_H} N_z$.
Being spherical, $X$ contains finitely many $B$-orbits by~\cite[Theorem~1]{Vin} (see also~\cite{Bri86}), hence so do~$X_f$ and~$\mathcal N$.
Then Proposition~\ref{prop_orb_corr} implies that $B_H$ has finitely many orbits in~$N_z$, so they are all stable under homotheties.
Consequently, minimal $B_H$-orbits in $N_z \setminus \lbrace 0 \rbrace$ are in bijection with minimal $B_H$-orbits in $\PP(N_z)$, which are precisely the closed $B_H$-orbits in~$\PP(N_z)$.
Being the product of $B_H^0$ with the center of~$H$, the group $B_H$ is solvable, hence all its orbits are affine.
Thus every closed $B_H$-orbit in~$\PP(N_z)$ is the union of a finite number of points.
Each such point $x$ is fixed by $B_H^0$ and hence corresponds to a $B_H^0$-semiinvariant vector $v \in N_z \setminus \lbrace 0 \rbrace$.
Being a highest-weight vector of a simple $H^0$-submodule in~$N_z$, $v$ is automatically $B_H$-semiinvariant, so $x$ is in fact fixed by~$B_H$.
Thus the minimal $B_H$-orbits in $\PP(N_z)$ are precisely the $B_H$-fixed points, so the minimal $B_H$-orbits in $N_z \setminus \lbrace 0 \rbrace$ are precisely those of nonzero $B_H$-semiinvariant vectors.
Applying Proposition~\ref{prop_orb_corr} we get the claim.
\end{proof}

Observe that the character restriction map $\mathfrak X(B) \to \mathfrak X(B_H)$ is surjective.

\begin{lemma} \label{lemma_B-section}
Suppose that $v \in N_z \setminus \lbrace 0 \rbrace$ is a $B_H$-semiinvariant vector of weight~$\chi_H \in \mathfrak X(B_H)$.
Then, for every character $\chi \in \mathfrak X(B)$ with $\left. \chi \right|_{B_H} = \chi_H$, there exists a unique section $s$ of~$\mathcal N$ such that $s(z) = v$ and $s$ is $B$-semiinvariant of weight~$\chi$.
\end{lemma}

\begin{proof}
The required section $s$ is given and uniquely determined by the formula $s(bz) = \chi(b^{-1})bz$ for all $b \in B$.
\end{proof}

\begin{theorem} \label{thm_main}
Suppose that $Y \subset X^{\operatorname{reg}}$.
Then for every minimal $G$-orbit $Y' \ne Y$ in~$X$ with the property $\overline{Y'} \supset Y$ there exists a $B$-root subgroup on $X$ that preserves $\partial X$ and connects $Y$ with~$Y'$.
\end{theorem}

\begin{proof}
Let $Y' \ne Y$ be a $G$-orbit in~$X$ that is minimal with the property $\overline {Y'} \supset Y$.
Then the intersection $Y' \cap X_f$ contains a $B$-orbit $O$ that is minimal in $X_f \setminus Y_f$.
Recall a $P$-equivariant isomorphism $X_f \simeq \mathcal N$ given by Proposition~\ref{prop_hfb_lst}(\ref{prop_hfb_lst_b}).
Thanks to Proposition~\ref{prop_B-orb_in_N}, under this isomorphism $O$ corresponds to the $B$-orbit of a $B_H$-semiinvariant vector $v \in N_z$.
By Lemma~\ref{lemma_B-section}, there exists a $B$-semiinvariant section $s$ of~$\mathcal N$ such that $s(z) = v$.
Then the $B$-root subgroup on~$\mathcal N$ defined by~$s$ as in Proposition~\ref{prop_RS_via_sect} connects $0$ with~$v$, hence the corresponding $B$-root subgroup on~$X_f$ connects $Y_f$ with~$O$ and thus $Y$ with~$Y'$.
By Lemma~\ref{extend} there exists a $B$-root subgroup on the whole $X$ that preserves $\partial X$ and connects $Y$ with~$Y'$.
\end{proof}

Recall from~\cite[Proposition~2.6]{AA} that for every $B$-root subgroup $R$ on~$X$ there is at most one $B$-stable prime divisor in~$X$ moved by~$R$.
The next proposition, which provides a partial converse to Theorem~\ref{thm_main}, is valid for any $G$-orbit $Y \subset X \setminus \partial X$ (not necessarily contained in~$X^{\operatorname{reg}}$).

\begin{proposition}
Suppose that $R$ is a $B$-root subgroup on~$X$ such that $\partial X$ is $R$-stable, all $B$-stable prime divisors in~$X$ not containing~$Y$ are $R$-stable, and $\overline Y$ is $R$-unstable.
Then there is a $P$-orbit $Y' \subset X_f \setminus Y_f$ such that $R$ connects $Y$ with~$Y'$ and $Y_f \cup Y'$ is closed in~$X_f$.
\end{proposition}

\begin{proof}
Put $B_L = B \cap L$ and $U_L = U \cap L$.
As $D \subset X$ is a union of $B$-stable prime divisors and none of them contains~$Y$, the hypotheses imply that $D$ is $R$-stable, hence so is~$X_f$.
Recall from Proposition~\ref{prop_Y_spherical}(\ref{prop_Y_spherical_a}) that $Y_f$ is the open $B$-orbit in~$Y$, so that $\overline{Y_f} = \overline Y$.
Since $R Y \not\subset \overline Y$, we have $R \overline Y \not\subset \overline Y$, that is, $R \overline{Y_f} \not \subset \overline{Y_f}$.
Thus $RY_f \not \subset \overline{Y_f}$ and by the discussion in~\cite[\S\,2.5]{AA} we have $RY_f = Y_f \cup O$ for some $B$-orbit $O \subset X \setminus \overline Y$ such that $\dim O = \dim Y_f + 1$ and $Y_f \subset \overline{O}$.
Clearly, $O \subset X_f$, hence the set $Z_O = Z \cap O$ is a single $B_L$-orbit.
Consider the natural projection map $\pi_{Z}\colon X_f  \simeq P_u \times Z \rightarrow Z$.
As $R$ commutes with~$U$, it permutes $P_u$-orbits in~$X_f$, hence the action of $R$ on~$X_f$ descends to~$Z$ along~$\pi_Z$ and thereby yields a $B_L$-root subgroup on~$Z$, which will be denoted by~$R'$.
Note that $R'Z_Y = Z_Y \cup Z_O$ and $\dim Z_O = \dim Z_Y + 1$.
Consider the $U_L$-fixed point set $Z^{U_L} \subset Z$.
Thanks to Proposition~\ref{prop_Y_spherical}(\ref{prop_Y_spherical_b},\,\ref{prop_Y_spherical_c}), $Z_Y$ is a single $T$-orbit and $Z_Y \subset Z^{U_L}$.
As $R'$ commutes with~$U_L$, the set $Z^{U_L}$ is $R'$-stable.
It follows that $Z_O \subset Z^{U_L}$ and hence $Z_O$ is also a single $T$-orbit.
Consider the closure $\overline{Z_O}$ of $Z_O$ in~$Z$; this is an irreducible affine subvariety of~$Z$.
Since $\overline{Z_O}$ contains an open $T$-orbit, it has a unique closed $T$-orbit, which is necessarily~$Z_Y$.
Then the relation $\dim Z_O = \dim Z_Y + 1$ yields $\overline{Z_O} = Z_Y \cup Z_O$, which implies that the union $Y_f \cup O$ is a closed subset of~$X_f$.
By~\cite[Proposition~2.7]{Tim}, the subset $P(Y_f \cup O) = Y_f \cup PO$ is also closed in~$X_f$, so the $P$-orbit $PO$ has the desired properties.
\end{proof}

\end{document}